\documentclass{article}
\usepackage{graphicx} % Required for inserting images
\usepackage[utf8]{inputenc}
\usepackage{amsmath,amssymb,amsthm}
\usepackage[english]{babel}
\usepackage[T2A]{fontenc}

\newtheorem{theorem}{Theorem}
\theoremstyle{definition}
\newtheorem{defi}{Definition}
\newtheorem{example}{Example}
\newtheorem{rem}{Remark}

\newcommand{\N}{\mathbb N}
\newcommand{\Z}{\mathbb Z}

\title{The groups $\Gamma_{n}^{4}$, braids, and $3$-manifolds}
\author{V.O.Manturov\footnote{This text is a result of discussion
with L.H.Kauffman,, E.A.Mudraya, and N.D.Shaposhnik. I am deeply indebted
to them for creating some parts of this text}
,\ I.M.Nikonov}
\date{April 2023}

\begin{document}

\maketitle

\begin{abstract}
We introduce a family of groups $\Gamma_n^k$ for integer parameters $n>k.$  These groups originate from discussion of braid groups
on $2$-surfaces. On the other hand, they turn out to be related to 3-manifolds (in particular, they lead to new relationships between braids
and manifolds), triangulations (ideal triangulations) cluster algebras, dynamics of moving points, quivers, hyperbolic structures,
tropical geometry, and, probably, many other areas still to be discovered.

Among crucial reason of this importance of groups $\Gamma_{n}^{4}$ we mention the Ptolemy relation, Pentagon relation, cluster algebra, Stasheff polytope.
\end{abstract}



\section{Introduction}
In \cite{KM}, the author introduced a family of groups $\Gamma_{n}^{4}$ depending on an integer parameter
$n>5$. These groups naturally describe motions of points on the plane (or a $2$-surface) by means of Vorono{\"\i} tiling and Delaunay
triangulations. As points pass through non-generic moments, generators of $\Gamma_{n}^{k}$ appear: they correspond to
flips of triangulations (Fig.~\ref{fig:flip}).

\begin{figure}[h]
\centering\includegraphics[width = 0.4\textwidth]{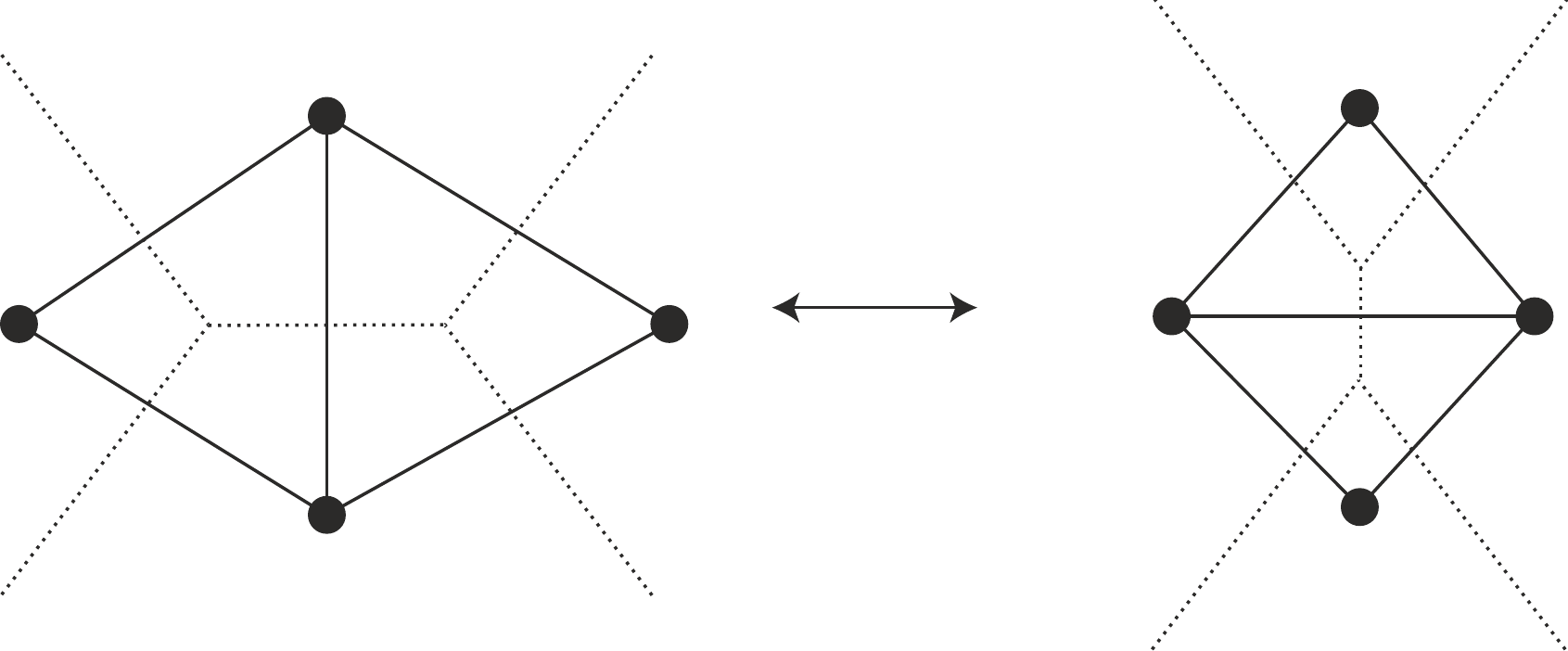}
\caption{A flip}\label{fig:flip}
\end{figure}

%{\bf Here one can draw a picture for $\Gamma$, say, from my slides: points---> Vorono{\"\i}---> Delaunay-->Points move --->flips}

The relations of the group $\Gamma$ are arranged in such a way that {\em isotopic braids give rise to equal elements
of the group $\Gamma$.} In other words, there is a homomorphism from the pure braid group on $n$ strands to the group $\Gamma_{n}^{4}$.
It seems likely that the kernel of the map $PB_{n}\to \Gamma_{n}^{4}$ consists of full twists (centre of the pure braid group), cf.~\cite{M22}.
Nevertheless, the map is very far from being surjective: there are many words in the alphabet of $\Gamma$ which do not originate
from any braids. Understanding the full depth of the groups $\Gamma_{n}^{4}$ (in particular, topological) is
a subject of special investigation.

In \cite{FKMN}, some presentations of the braid groups by means of $\Gamma$ are obtained.

The most interesting relation of the group $\Gamma_{n}^{4}$ is the {\em pentagon relation}. Graphically it looks as
shown in Fig.~\ref{fig:gamma4_relations} right. This relation says that one can apply five consecutive flips to a given triangulation of a pentagon
and return to the initial triangulation.

The pentagon relation appears in various branches of mathematics under various names. Just to mention a few of them:
{\em cluster algebras}, {\em triangulation of polytopes}, {\em Teichm\"uller spaces}, {\em Associahedra (Stasheff polytope)}, {\em recoupling theory
for $3$-manifolds} (as Biedenharn--Elliot identity).

This naturally suggests that all those parts of mathematics should be related to each other.
Some of them indeed are: say, in \cite{GKZ}, the authors found very deep connections between possible
triangulations of $2$-surfaces, cluster algebras and ...
View Figure~\ref{fig:cl_alg}.

\begin{figure}[h]
\centering\includegraphics[width = 0.9\textwidth]{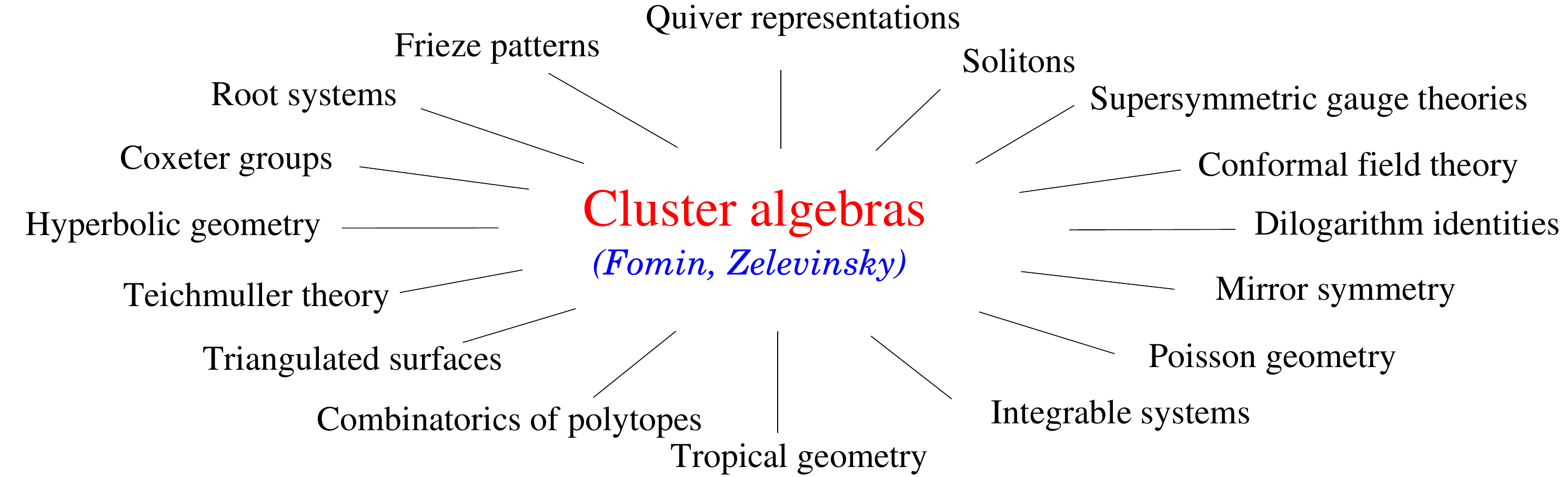}
\caption{Cluster algebras connections and applications~\cite{Felikson}}\label{fig:cl_alg}
\end{figure}

However, it seems that the authors overlooked the {\em group structure}. Once one gets this group structure,
one immediately gets invariants of braids as shown in \cite{FKMN}.
Hence, potentially all structures listed above should have something to do with braid groups (not
only classical braid groups for the plane but also braid groups for $2$-surfaces).

One exception is the {\em groupoid structure} \cite{FC} related to {\em cluster algebras} and
{\em Teichm\"uller  spaces}. This relation seems quite natural since Teichm\"uller spaces
are related to mapping class groups and braid groups. The relations between braid groups
these groupoids will be discussed elsewhere.

On the other hand, from the point of view of a low-dimensional topologist, it is worth mentioning that more
or less the same equation appears in another form (actually, the tensor form) in
the study of $3$-manifold invariants (Turaev--Viro invariants~\cite{Matveev, TV}).

The aim of the present paper is to establish a deeper connection between the subjects listed above.

The two concrete applications we touch on in more detail, are:

\begin{enumerate}

\item How to use recoupling theory (previously used for constructing $3$-manifold invariants~\cite{KL})
to get invariants of braids.

A tiny piece of this research was sketched in \cite{MN}.

\item How to use constructions of braid invariants to get Turaev--Viro like invariants of $3$-manifolds.

\end{enumerate}

Of course, we don't pretend to cover all possible relations in this single paper.
In particular, we don't touch on {\em higher groups $\Gamma_{n}^{k},k\ge 5$},
introduced by I.M.Nikonov in the book \cite{FKMN}. These groups describe triangulations of higher
dimensional manifolds; again, similar structures from other points of view, and again,
the group structure was overlooked in many deep and serious papers.

The paper is organised as follows. In Section~\ref{sect:gamma4n} we define the groups $\Gamma_n^4$ and in Section~\ref{sect:braid_to_gamma} describe the homomorphism from the braid group to $\Gamma_n^4$. In Section~\ref{sect:ptolemy_relation} we remind Ptolemy relation. Section~\ref{sect:recoupling_theory} we describe invariants of braids which come from the recoupling theory. In Section~\ref{sect:special_spines} we remind special spines of $3$-manifolds. In Section~\ref{sect:3manifold_invariants_from_gamma} we describe invariants of $3$-manifolds based on Ptolemy relation. Section~\ref{sect:ubiquitous_gamma} concludes the paper with the discussion of possible further applications of the groups $\Gamma$.

\subsection*{Acknowledgements}

We are extremely grateful to Louis H.Kauffman, Vladimir G.Turaev, Victor M.Buchstaber, Elizaveta A.Mudraya, Nikita D.Shaposhnik, Zheyan Wan for various stimulating discussions.

Some relations of the pentagon equation (which also lead to representations of various braid groups)
where pointed out to the first named author by Igor G.Korepanov.
We are extremely grateful to him for pointing it out and for the joy the first author shared with
him when working on this project. Korepanov's articles \cite{KS99, KS13, K14}
should give rise to many further deep connections of the group $\Gamma_{n}^{k}$ (not
only $\Gamma_{n}^{4}$); these results should appear elsewhere.

\section{The groups $\Gamma_{n}^{4}$}\label{sect:gamma4n}

The groups $\Gamma_{n}^{4}$ were introduced in~\cite{KM, FKMN, FMN, FKMN1} with the aim to give an algebraic description of the space of triangulations with a fixed number of
vertices.

\begin{defi}\label{dfn_Gamma}
The group $\Gamma_{n}^{4}$\index{Group!$\Gamma_n^4$}
is given by generators
 $\{ d_{(ijkl)}~|~ \{i,j,k,l\} \subset \bar{n}, |\{i,j,k,l\}| = 4\}$
and relations:

\begin{enumerate}
\item $d_{(ijkl)}^{2} = 1$ for $(i,j,k,l) \subset \bar{n}$,
\item $d_{(ijkl)}d_{(stuv)} = d_{(stuv)}d_{(ijkl)}$, for $| \{i,j,k,l\} \cap \{s,t,u,v\} | < 3$,
\item $d_{(ijkl)}d_{(ijlm)}d_{(jklm)}d_{(ijkm)}d_{(iklm)} = 1$ for distinct $i,j,k,l,m$.
\item $d_{(ijkl)}=d_{(jilk)}=d_{(klij)}=d_{(kjil)}=d_{(ilkj)}=d_{(jkli)}=d_{(lkji)}=d_{(lijk)}$ for distinct $i,j,k,l$.
\end{enumerate}
\end{defi}

\begin{figure}[h]
\centering\includegraphics[width = 0.4\textwidth]{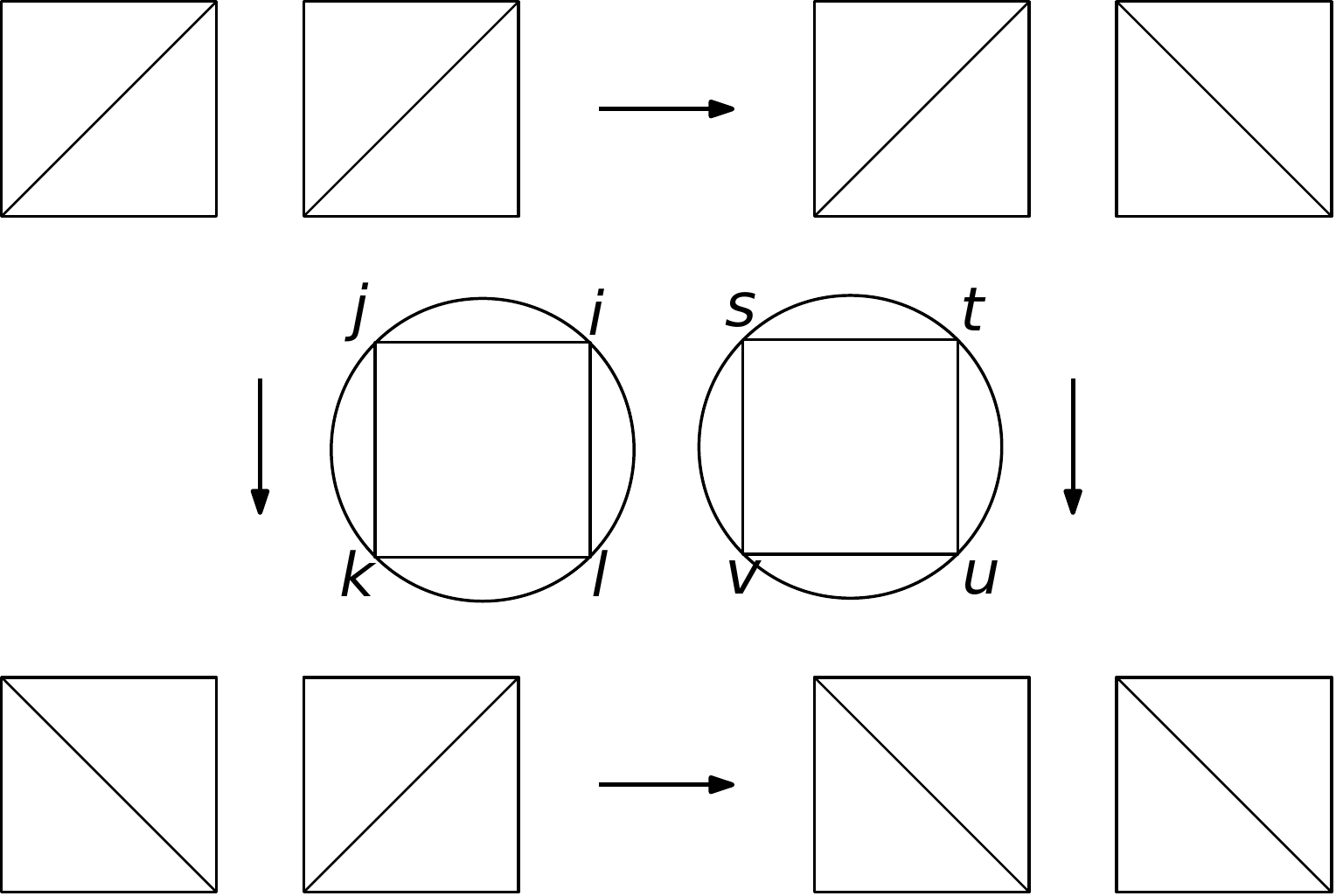}\qquad\includegraphics[width = 0.35\textwidth]{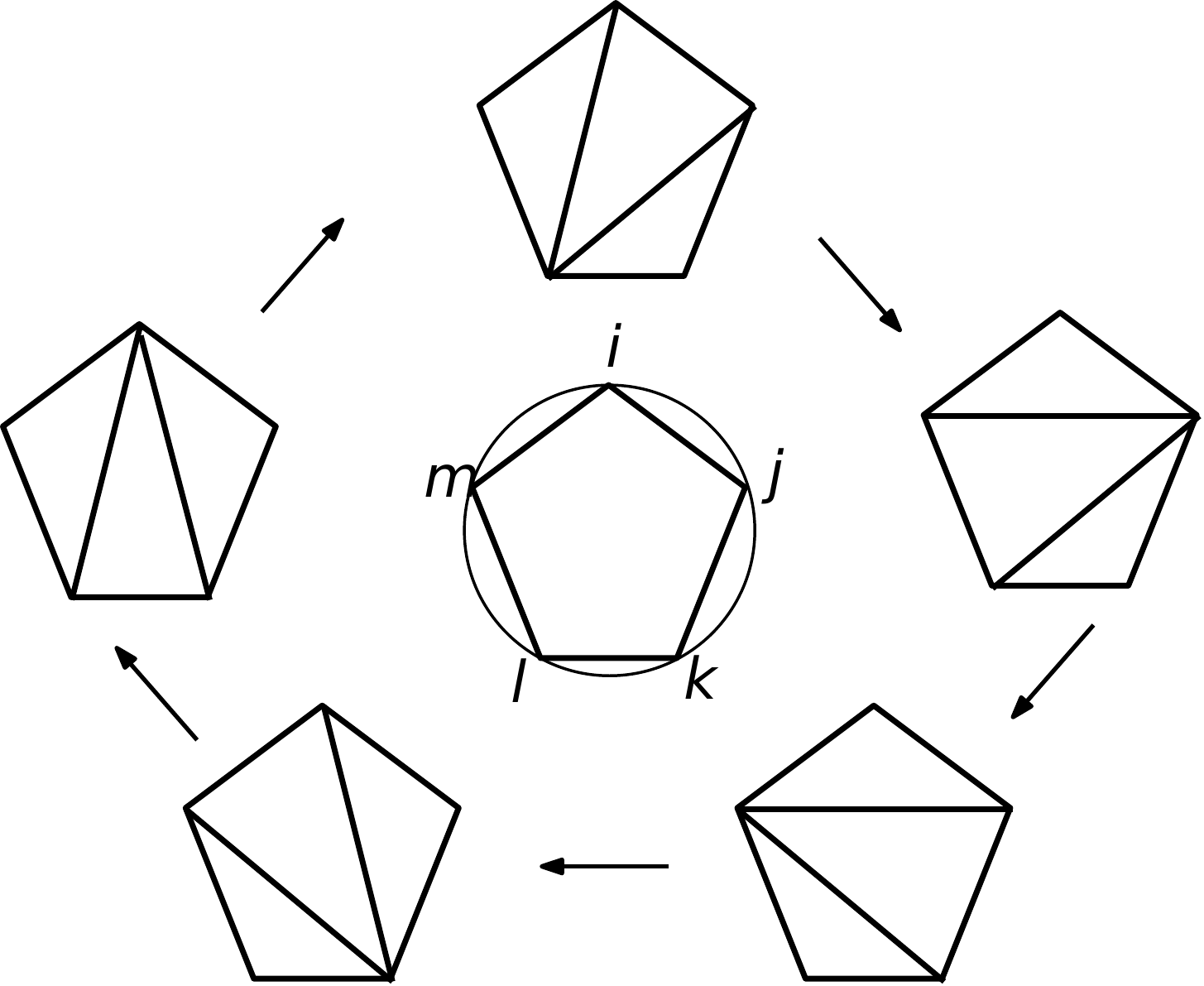}
\caption{Sequences of flips corresponding to relations in $\Gamma_{n}^{4}$}\label{fig:gamma4_relations}
\end{figure}

\subsection{Where do groups $\Gamma$ come from?}

We consider a braid: $n$ pairwise  distinct points move on the plane (or two-surface). They form a Vorono{\"\i} tiling and its dual Delaunay triangulation; as they move generically they undergo flips we write down generators corresponding to flips. Thus for each generic braid we get a word.

{\bf Relations from groups $\Gamma$ are such that isotopic braids yield equal words.

}

\vspace{0.1cm}
{\bf Known fact:}
Pentagon relation yields a braid group action on labeled triangulation.(every edge is labeled by a Laurent polynomials.
\vspace{0.1cm}
As we shall see,
the groups
$\Gamma$
lead to invariants of $3$-manifolds.

\vspace{0.1cm}
and the techniques used for constructing invariants of 3-manifolds,
can lead to finite-dimensional representations of $\Gamma$.

\section{The map from pure braid groups to $\Gamma_{n}^{4}$}\label{sect:braid_to_gamma}

In this section we describe the group homomorphism $f_{n}$ from the pure braid group $PB_{n}$ to $\Gamma_{n}^{4}$~\cite{KM}. The topological background for that is very easy: we consider codimension 1 ``walls'' which correspond to generators (flips) and codimension 2 relations (of the group $\Gamma$). Having this, we construct a map on the level of generators.

Let us use the notation $\bar{n} :=\{1, \cdots n\}$.

\begin{defi}
The pure braid group $PB_{n}$ of $n$ strands is the group given by group presentation generated by $\{ b_{ij} | i,j \in \bar{n}, i<j\}$ subject to the following relations:
\begin{enumerate}
\item $b_{ij}b_{kl} = b_{kl}b_{ij}$ for $i,j,k,l \in \bar{n} $ such that $i<j<k<l$ or $i<k<l<j$;
\item $b_{ij}b_{ik}b_{jk} = b_{ik}b_{jk}b_{ij} = b_{jk}b_{ij}b_{ik}$ for $i,j,k \in \bar{n} $ such that $i<j<k$;
\item  $b_{ik}b_{jk}b_{jl}b_{jk} = b_{jk}b_{jl}b_{jk}b_{ik} $ for $i,j,k,l \in \bar{n} $ such that $i<j<k<l$.
\end{enumerate}
\end{defi}

The mapping $f_{n} : PB_{n} \rightarrow \Gamma_{n}^{4}$ can be formulated as follows:
Let us denote
\begin{center}
$d_{\{p,q, (r,s)_{s}\}}  = \left\{
\begin{array}{cc} % brackets may be (...), [...], \{...\}, or left out
     d_{(pqrs)} & \text{if}~p<q<s, \\
      d_{(prsq)} & \text{if}~p<s<q, \\
  d_{(rspq)} & \text{if}~s<p<q,\\
  d_{(qprs)} & \text{if}~q<p<s, \\
      d_{(qrsp)} & \text{if}~q<s<p, \\
  d_{(rsqp)} & \text{if}~s<q<p.
   \end{array}\right.$
   \end{center}
 \begin{rem}
Notice that the generator $d_{\{p,q, (r,s)_{s}\}}$ corresponds to four points $P_{p},P_{q},P_{r},P_{s}$ such that they are placed on a circle according to the order of $p,q,s$ and the point $P_{r}$ is placed close to $P_{s}$ for the orientation $P_{r}$ to $P_{s}$ to be the counterclockwise orientation, see Fig.~\ref{exa_pts_circle}. The subscription $s$ of $(r,s)_{s}$ means that the point $P_{s}$ does not move, but the point $P_{r}$ will move turning around the point $P_{s}$ after this moment. In other words, when we use the notation $d_{\{p,q, (r,s)_{s}\}}$, we are looking that the point $P_{r}$ is ``moving'' closely to the point $P_{s}$, turning around $P_{s}$. We would like to highlight that $d_{\{p,q, (r,s)_{s}\}} \neq d_{\{p,q, (s,r)_{s}\}}$.

\begin{figure}[h!]
 \centering
 \includegraphics[width = 0.5\textwidth]{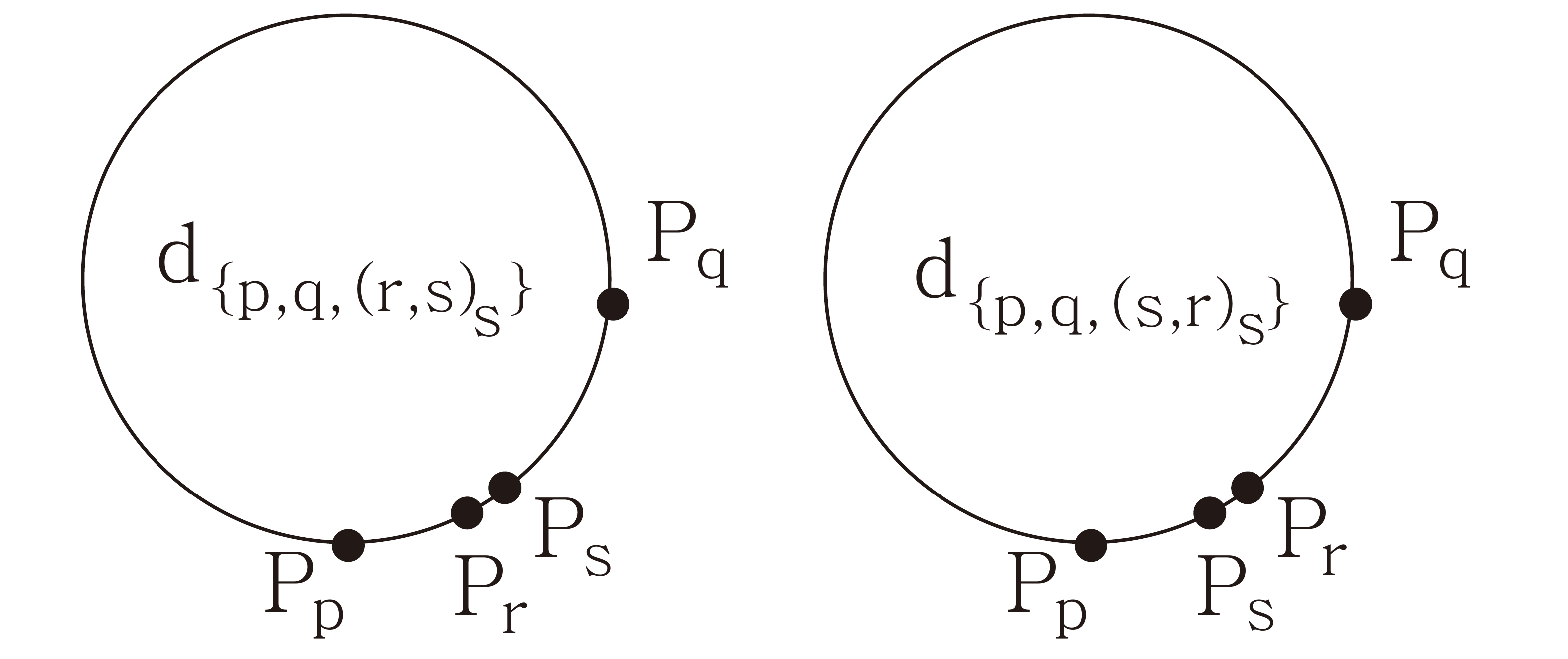}
 \caption{For $p<s<q$, $d_{\{p,q, (r,s)_{s}\}} = d_{(prsq)}$, but $d_{\{p,q, (s,r)_{s}\}} = d_{(psrq)}$.}\label{exa_pts_circle}
\end{figure}
\end{rem}

Let $b_{ij} \in PB_{n}$, $1 \leq i<j \leq n$, be a generator. Consider the elements
\begin{eqnarray*}
\gamma_{i,(i,j)}^{I} = \prod_{p=2}^{j-1}\prod_{q=1}^{p-1}d_{\{p,q,(i,j)_{j}\}},\\
\gamma_{i,(i,j)}^{II} = \prod_{p=1}^{j-1}\prod_{q=1}^{n-j}d_{\{(j-p),(j+p),(i,j)_{j}\}},\\
\gamma_{i,(i,j)}^{III} = \prod_{p=1}^{n-j+1}\prod_{q=0}^{n-p+1}d_{\{(n-p),(n-q),(i,j)_{j}\}},\\
\gamma_{i,(i,j)}=\gamma_{i,(i,j)}^{II}\gamma_{i,(i,j)}^{I}\gamma_{i,(i,j)}^{III}.
\end{eqnarray*}
Now we define $f_{n} : PB_{n} \rightarrow \Gamma_{n}^{4}$ by
$$f_{n}(b_{ij}) = \gamma_{i,(i,(i+1))}\cdots \gamma_{i,(i,(j-1))}\gamma_{i,(i,j)}\gamma_{i,(j,i)}\gamma_{i,((j-1),i)}^{-1}\cdots \gamma^{-1}_{i,((i+1),i)},$$
for $1 \leq i<j \leq n$.

\begin{theorem}[\cite{KM}]\label{thm_to_gamma}
The map $f_{n} : PB_{n} \rightarrow \Gamma_{n}^{4}$, which is defined above, is a well defined homomorphism.
\end{theorem}

\section{The Ptolemy relation and the pentagon relation}\label{sect:ptolemy_relation}

For an inscribed quadrilateral with consequent edge-lengths a,b,c,d and two diagonals x, y one has: $ac + bd = xy$.

\begin{figure}[h]
\includegraphics[width=2.9cm]{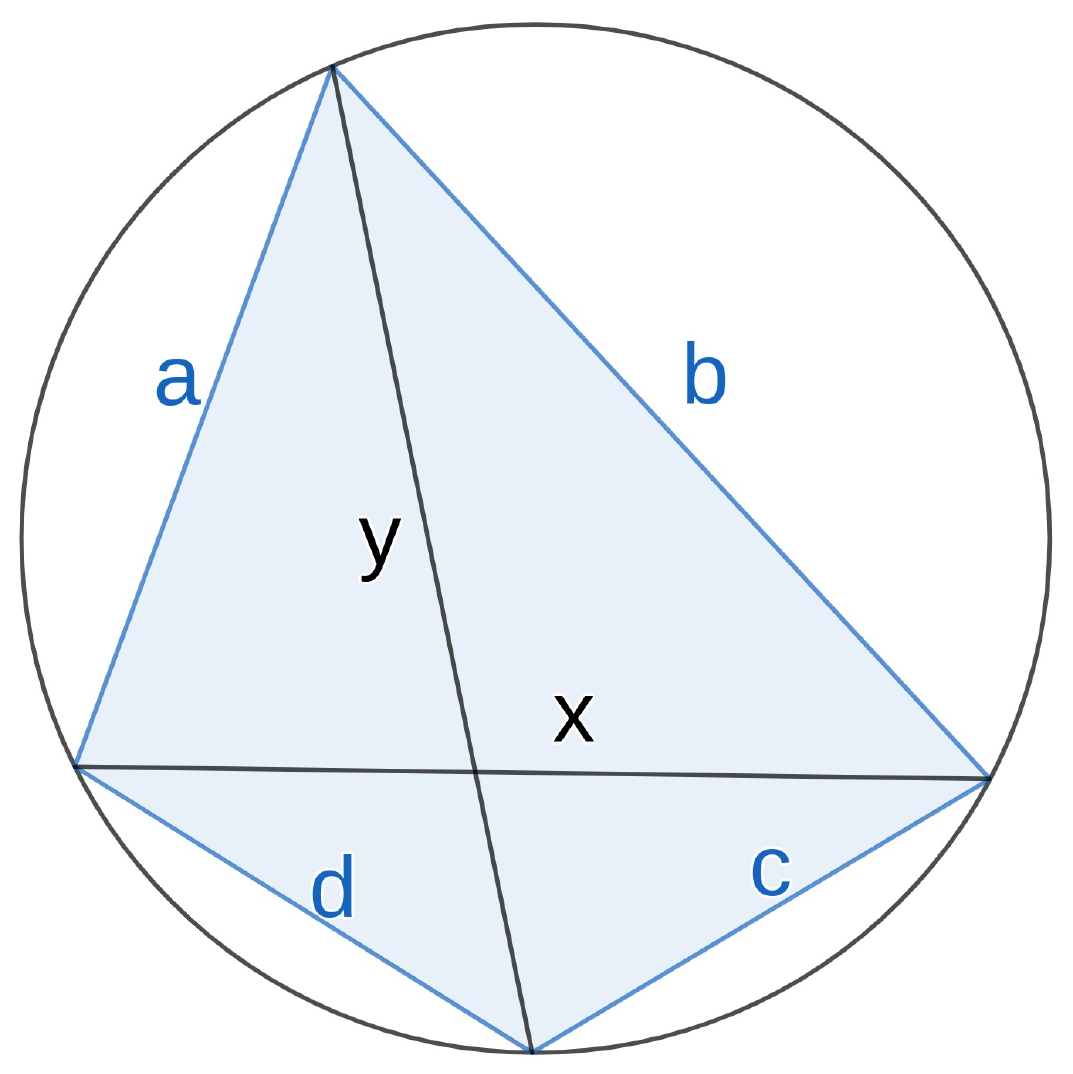}
\centering
\caption{An inscribed quadrilateral}
\label{quadrilateral}
\end{figure}

Another way to thinking of this relation is as follows. We enumerate the vertices of the quadrilateral by numbers from  $1$ to $4$ and associate variable $x_{ij}$ with the edge oriented from $i$ to $j$, so that the opposite edge is assigned $x_{ji} = -x_{ij}$.
Then the Ptolemy relation looks like:
\begin{equation}\label{eq:ptolemy}
x_{12}x_{34} + x_{23}x_{14} + x_{13}x_{42} = 0.
\end{equation}
This equation is invariant under permutation of indices $(1, 2, 3, 4)$, hence it is valid also for self-intersecting quadrilaterals.

{The signed Ptolemy relation}

\begin{figure}[h]
\includegraphics[width=0.5\textwidth]{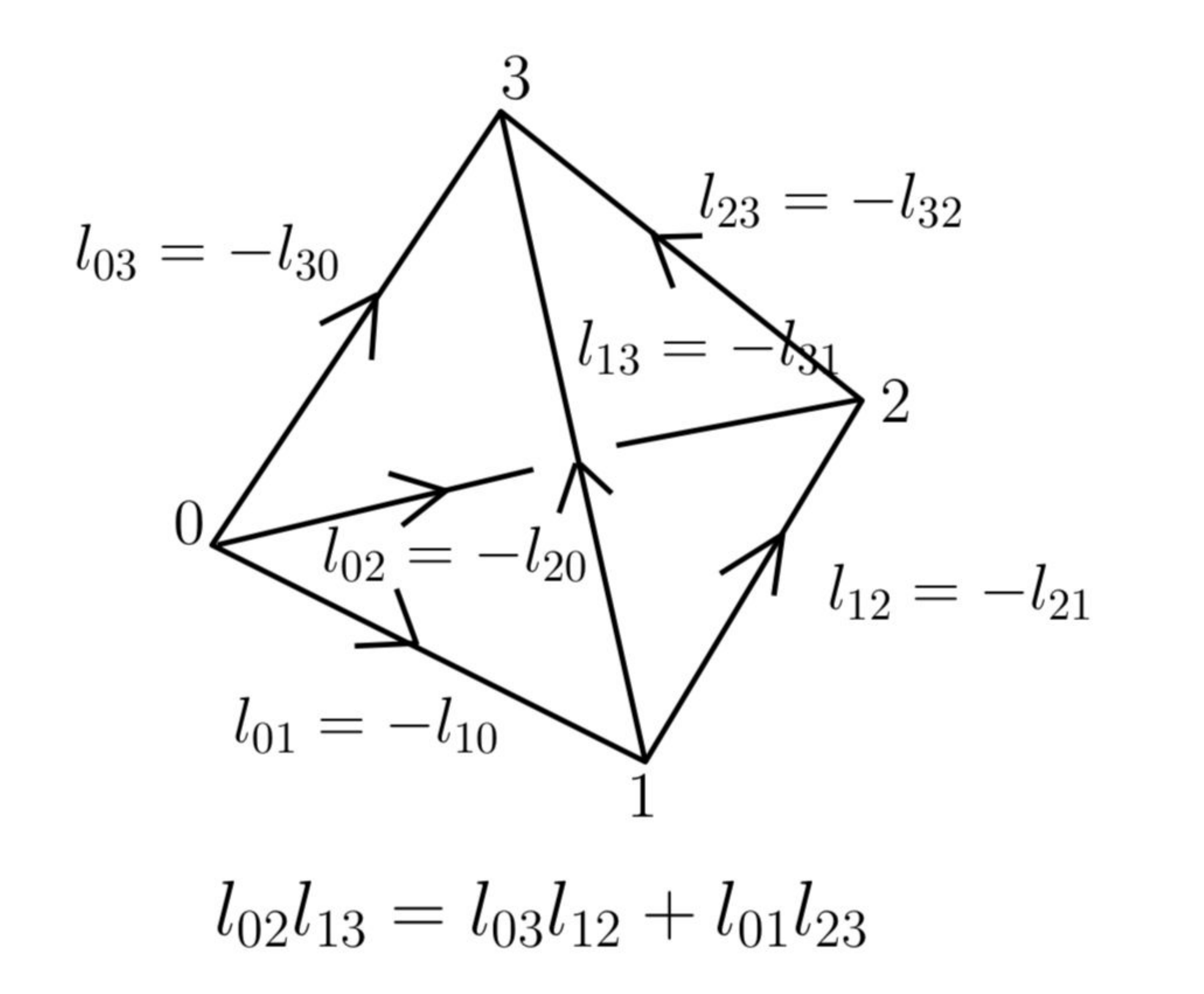}
\centering
\caption{An inscribed quadrilateral}
\label{quadrilateral}
\end{figure}

The label corresponds to an oriented edge. Inverting the edge will change the
sign of the label.

The ``oriented'' Ptolemy relation defined as above is invariant under any permutation of
indices $1,2,3,4$.

\section{The Pentagon relation}\label{sect:pentagon_relation}
There are exactly five triangulations of a pentagon with no additional points.

Each of these triangulations contains two non-intersecting (adjacent) diagonals.

For each triangulation, one can flip one of the diagonals in question.

Performing the five consequent flips, we return to the initial position.

So, schematically, the pentagon relation looks like

%<здесь картинку с пятиугольниками>

This relation was known to specialists in cluster algebras~\cite{FZ1, FZ2}.

For us, it is the main relation in groups $\Gamma_{n}^{4}$.

The geometric meaning of this relation from the point of view of the braid group is
the following:

we consider five points almost on the same circle. The Delaunay triangulation
consists of three triangles tiling the pentagon.

As we move these points a little, the Delaunay triangulation undergo five
consecutive flips and since there is no braiding in this process, the element
of the group $\Gamma$ represented by the product of such flips has to
be trivial.

\section{Recoupling theory}\label{sect:recoupling_theory}

The essence of the recoupling theory \cite{KL,KR} is the following.

We draw trivalent graphs (on the plane or on 2-surface) and mark labels by integer numbers.

Whenever we see number $k$, we can think that there is a bunch of $k$ parallel strands along this edge.
At each vertex each of those strands turns right or left (Fig.~\ref{fig:strand_triangle}).

\begin{figure}[h]
\centering\includegraphics[width = 0.3\textwidth]{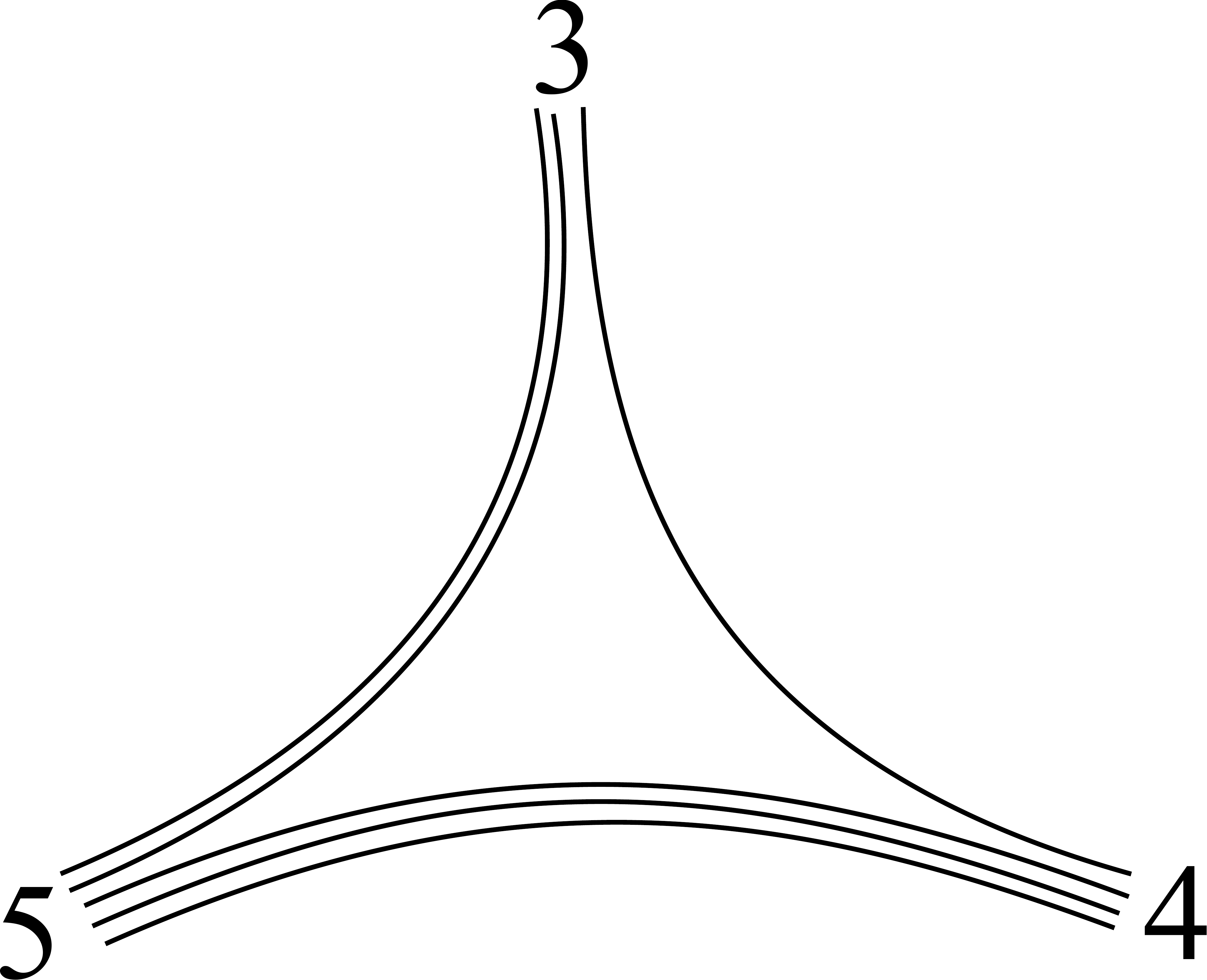}
\caption{Strands at a vertex}\label{fig:strand_triangle}
\end{figure}
This naturally yields the two conditions, the parity and the triangle inequality.

Besides, fix an integer $r\ge 3$ and restrict ourselves to the situation when
the numbers of strands $f(a),f(b),f(c)$ for three edges coming to one vertex satisfy
$f(a)+f(b)+f(c)\le 2r$. This leads to the notion of an {\em admissible colouring}.

As the time $t$ goes, the tree undergoes flips.

Fix an integer $r\ge 3$. Denote $q=e^{\frac{i\pi}r}\in\mathbb C$. For any regular value $t$, consider the set $F_r(t)=F_r(G(t))$ of admissible colourings of the edges of the Delaunay triangulation $G(t)$, i.e. maps $f: E(G(t))\to\N\cup\{0\}$ such that for any triangle of $G(t)$ with edges $a,b,c$ one has
\begin{enumerate}
\item $f(a)+f(b)+f(c)$ is even;
\item $f(a)+f(b)\ge f(c)$, $f(a)+f(c)\ge f(b)$, $f(b)+f(c)\ge f(a)$ ;
\item $f(a)+f(b)+f(c)\le 2r-4$.
\end{enumerate}
Let $V_r(t)$ be the linear space with the basis $F_r(t)$.

\begin{figure}[h]
\centering\includegraphics[width = 0.4\textwidth]{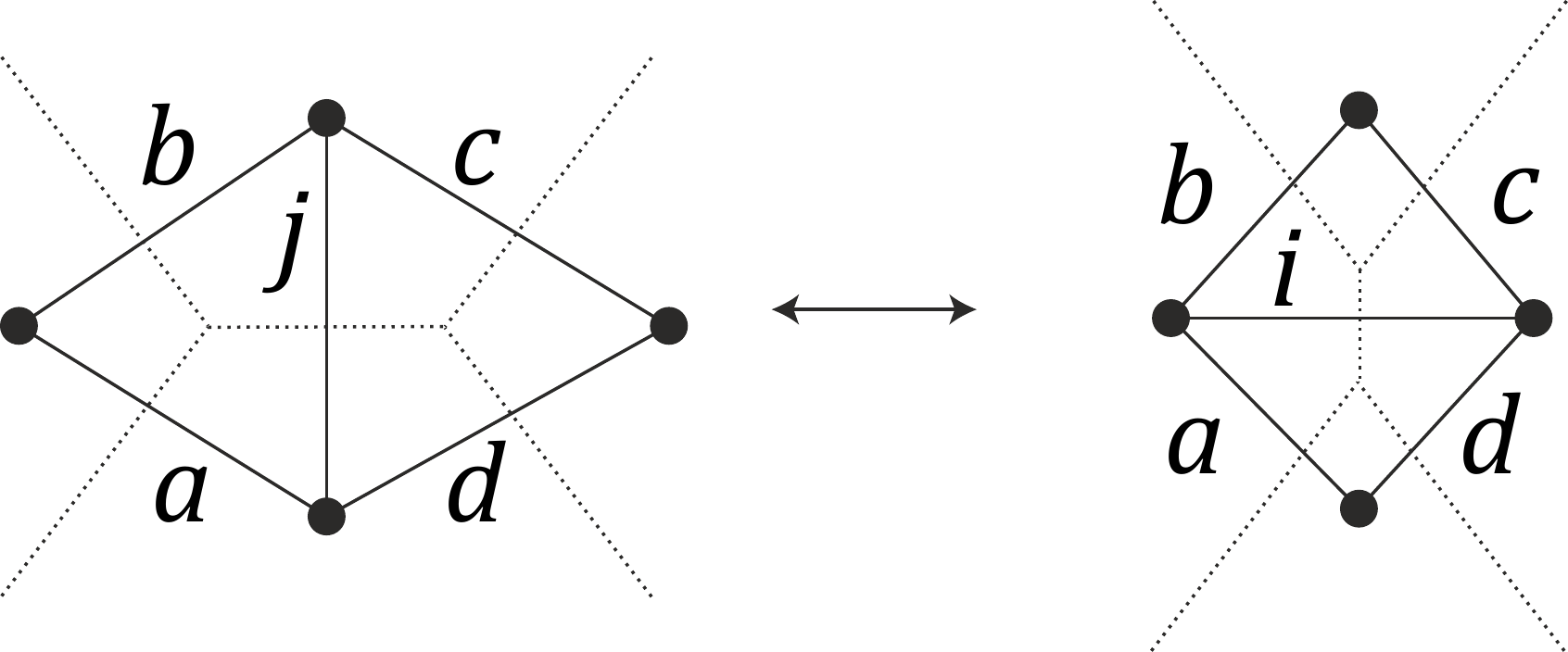}
\caption{Transformation of Delaunay triangulation}\label{tiling_change}
\end{figure}

Consider a singular value $t_k$. At the moment $t_k$ a transformation (flip) of the triangulation occurs (Fig.~\ref{tiling_change}). By recoupling theorem~\cite[Theorem 6.2]{KL} one can define a linear operator $A_r(t_k): V_r(X(t_k-0))\to V_r(X(t_k+0))$ whose coefficients are $q-6j$-symbols. More detailed, the triangulations $G(t_k-0)$ and $G(t_k+0)$ differ by an edge: $j=E(G(t_k-0))\setminus E(G(t_k+0))$, $i=E(G(t_k+0))\setminus E(G(t_k-0))$. Let $a,b,c,d$ be the edges incident to same triangles as $i$ or $j$. For any $f\in F_r(t_k-0)$ considered as a basis element of $V_r(t_k-0)$, we set
\[
A_r(t_k)f=\sum_{g\in F_r(t_k-0)\colon f=g|_{E(G(t_k-0))\cap E(G(t_k+0))}}
\left\{\begin{array}{ccc}
g(a) & g(b) & g(i)\\
f(c) & f(d) & f(j)
\end{array}\right\}_q\cdot g,
\]
where $\left\{\begin{array}{ccc}
\alpha & \beta & \xi\\
\gamma & \delta & \eta
\end{array}\right\}_q$ is the $q-6j$-symbol (for the definition see~\cite[Proposition 11]{KL}).

Let $A_r(\beta)=\prod_{k=1}^m A_r(t_k): V(X)\to V(X)$ be the composition of operators $A(t_k)$.

\begin{theorem}
The operator $A_r(\beta)$ is a pure braid invariant.
\end{theorem}

\begin{proof}
Consider a pure braid isotopy $\beta_s$, $s\in[0,1]$. We can assume the isotopy is generic. This means that the isotopy includes a finite number of values $s_l$ such that the braid $\beta_{s_l}$ has a unique value $t^*$ such that the set $X_s(t^*)$ contains one of the following singular configurations
\begin{enumerate}
\item five points which belong to one circle (or line);
\item two quadruples of points which lie on one circle. The circles for the quadruples are different.
\end{enumerate}
The first case leads to the pentagon relation for Delaunay triangulation transformations (Fig.~\ref{flip_5-gon}). This relation induces a pentagon relation for the operator $A_r(\beta)$ which holds by theorem~\cite[Proposition 10]{KL}.
\begin{figure}[h]
\centering\includegraphics[width = 0.6\textwidth]{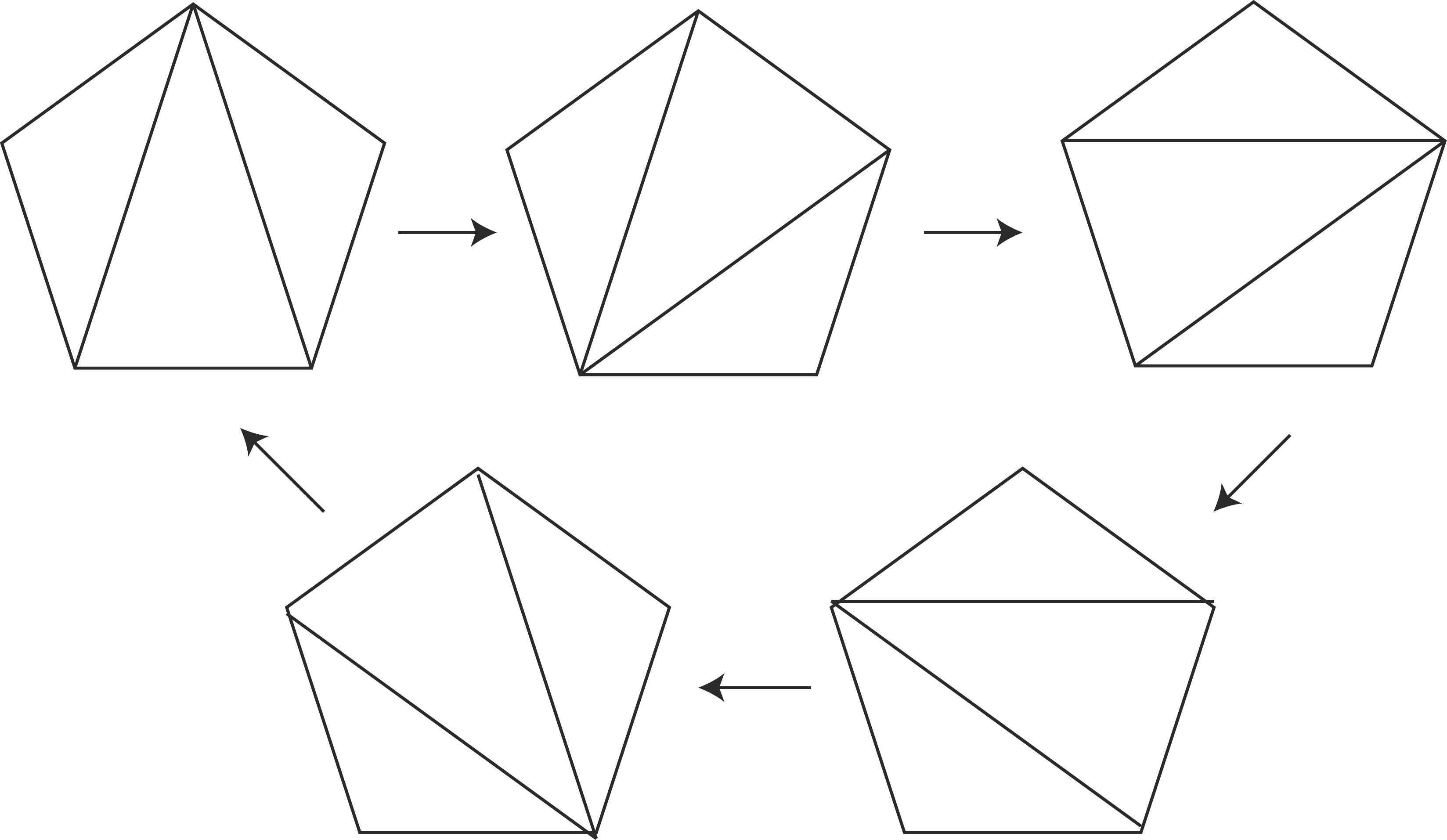}
\caption{Pentagon relation}\label{flip_5-gon}
\end{figure}

The second case leads to commutativity relation $A_1A_2=A_2A_1$ where $A_1$, $A_2$ correspond to flips on the given quadruples of points. The relation holds by construction of the operators $A_1$, $A_2$.
\end{proof}

Thus, we get a construction analogous to that in~\cite[Section 16.6]{FKMN} which uses Ptolemy relation.

\begin{example}
Consider the point configuration $P_1=(0,0)$, $P_2=(\frac 13,0)$, $P_3=(-\frac 12,\frac{\sqrt 3}2)$, $P_3=(-\frac 12,-\frac{\sqrt 3}2)$, $P_5=(1,0)$. Consider the pure braid $\beta$ given by dynamics of $P_2$ moving around $P_1$: $P_2(t)=(\frac 13\cos t, \frac 13\sin t)$, $t\in[0,2\pi]$. The dynamics produces six different Delaunay triangulations. For $r=4$, each triangulation has $160$ admissible colourings. The operator $A_4(\beta)$ of dimension $160$ is not identity (it has the eigenvalue $-1$ of multiplicity $20$). Thus, the invariant $A_r$ is not trivial.
\end{example}

\section{Special spines of $3$-manifolds}\label{sect:special_spines}

Here we shall give the definition of special spine (see~\cite{Matveev}).

\begin{defi}\label{def:subpolyhedron}
 An \emph{simple polyhedron} is a finite CW-complex $P$ such that each point $x$ in $P$ has a neighbourhood homeomorphic to one of the three configurations (see Figure~\ref{fig:spinePointTypes}):
\begin{enumerate}
   \item a disc;
    \item three discs intersecting along a common arc in their boundaries, where the point $x$ is on the common arc; or
    \item the \emph{butterfly}, which is a configuration built from four arcs that meet at the point $x$, with a face running along each of the six possible pairs of arcs.
\end{enumerate}
%Given an admissible polyhedron $\Sigma$, we write $\Sigma^{(0)}$ for the points with butterfly neighbourhoods, $\Sigma^{(1)}$ for the points with three-disc neighbourhoods, and $\Sigma^{(2)}$ for the points with disc neighbourhoods.
%We refer to the connected components of $\Sigma^{(0)}$, $\Sigma^{(1)}$ and $\Sigma^{(2)}$ as \emph{vertices}, \emph{edges} and \emph{faces}, respectively.

\begin{figure}[!ht]
    \centering
    \includegraphics[width=0.25\textwidth]{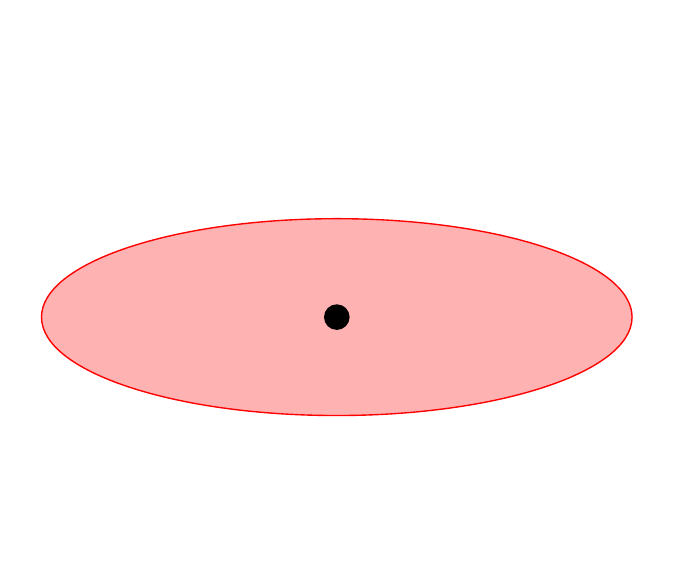}\quad
    \includegraphics[width=0.25\textwidth]{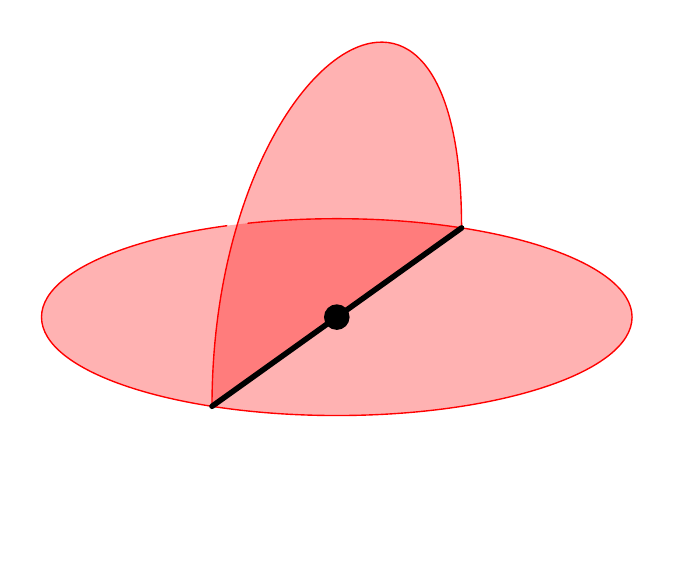}\quad
    \includegraphics[width=0.25\textwidth]{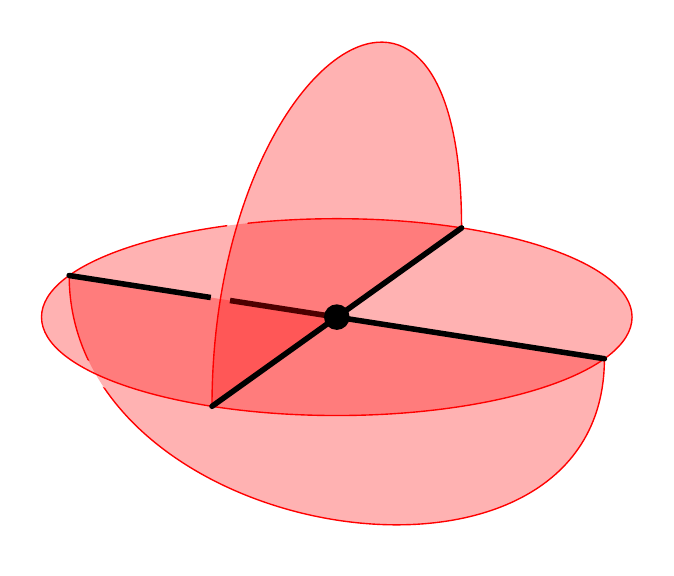}
    \caption{The three types of points in a simple polyhedron $P$.}
    \label{fig:spinePointTypes}
\end{figure}

Each simple polyhedron is naturally stratified. In this stratification each stratum of dimension $2$ (a $2$-component) is a connected component of the set of nonsingular points. Strata of dimension $1$ consist of open or closed triple lines, and dimension $0$ strata are true vertices.

A simple polyhedron $P$ is called \emph{special} if:
\begin{enumerate}
\item Each 1-stratum of $P$ is an open 1-cell.
\item Each 2-component of $P$ is an open 2-cell.
\end{enumerate}

A \emph{spine} of a compact connected $3$-manifold with boundary $M$ is a subpolyhedron $P\subset Int M$ such that $M\setminus P$ is homeomorphic to $\partial M\times [0,1)$. By a spine of a closed connected $3$-manifold $M$ we mean a spine of $M\setminus Int B^3$ where $B^3$ is a $3$-ball in $M$.

\end{defi}

\section{Turaev--Viro invariants}\label{sect:turaev_viro_invariants}
Consider a special spine of a $3$-manifold and associate
weights of butterflies to true vertices. These
weight depend on $6$ parameters and they are called ($q-6j$)-symbols).

\begin{figure}[h]
%\centering
%\begin{minipage}[t]{0.4\textwidth}
%\centering
%
%\includegraphics[width=1\linewidth]{MP1.pdf}
%
%\end{minipage}
%\begin{minipage}[t]{0.4\textwidth}
%\centering
%
%\includegraphics[width=1\linewidth]{MP2.pdf}
%\end{minipage}

\centering\includegraphics[width=0.8\textwidth]{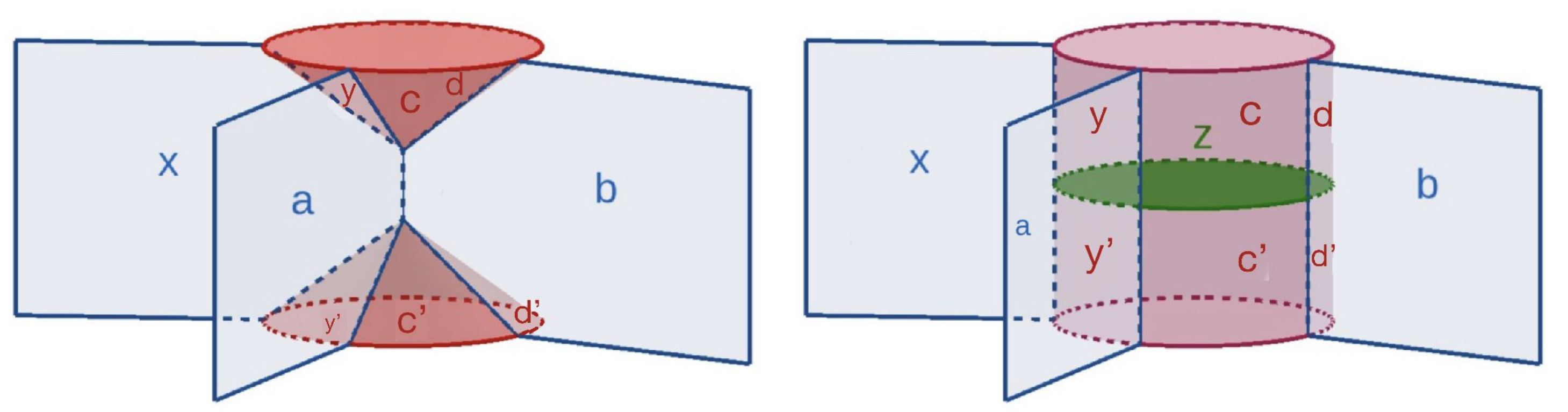}
\caption{The Matveev-Piergallini move}
\label{MP}

\end{figure}

\section{$(q-6j)$-symbols}\label{sect:q6j-symbols}

Each butterfly has $6$ wings which are naturally split into three pairs of {\em opposite ones}.
It is naturally to encode a butterfly by the $1$-frame of the $3$-simplex.

Hence, a coloured butterfly associates $6$ numbers (or elements of some ring).
Since each butterfly has many symmetries, it is natural to require that these assignments
are invariant under rotations. In other words, we associate numbers ($(q-6j)$-symbols)
\[
\left[\begin{array}{ccc} a & b & c\\ d & e & f\end{array}\right]
\]
to hexuples of integers $(a,b,c,d,e,f)$ in such a way that
\[
\left[\begin{array}{ccc} a & b & c\\ d & e & f\end{array}\right] = \left[\begin{array}{ccc} b & a & c\\ e & d & f\end{array}\right],\quad
\left[\begin{array}{ccc} a & b & c\\ d & e & f\end{array}\right] = \left[\begin{array}{ccc} f & e & a\\ c & b & d\end{array}\right].
\]

\section{Constructing invariants of $3$-manifolds from $\Gamma$}\label{sect:3manifold_invariants_from_gamma}

Below, we give a hint how to construct the invariant geometrically.

However, there is a tiny gap in the argument.

Motivated by this example, we present an algebraic description of the invariant.

For each triangle with sides $(a,b,x)$
and two inscribed quadrilaterals $(a,b,c,d),(a,b,c',d')$
with diagonals $y,y'$
respectively there exists a unique $z$
giving rise to three (self-intersecting)
quadrilaterals.

Solution: completing a pentagon.
\begin{figure}
\centering\includegraphics[width=0.35\textwidth]{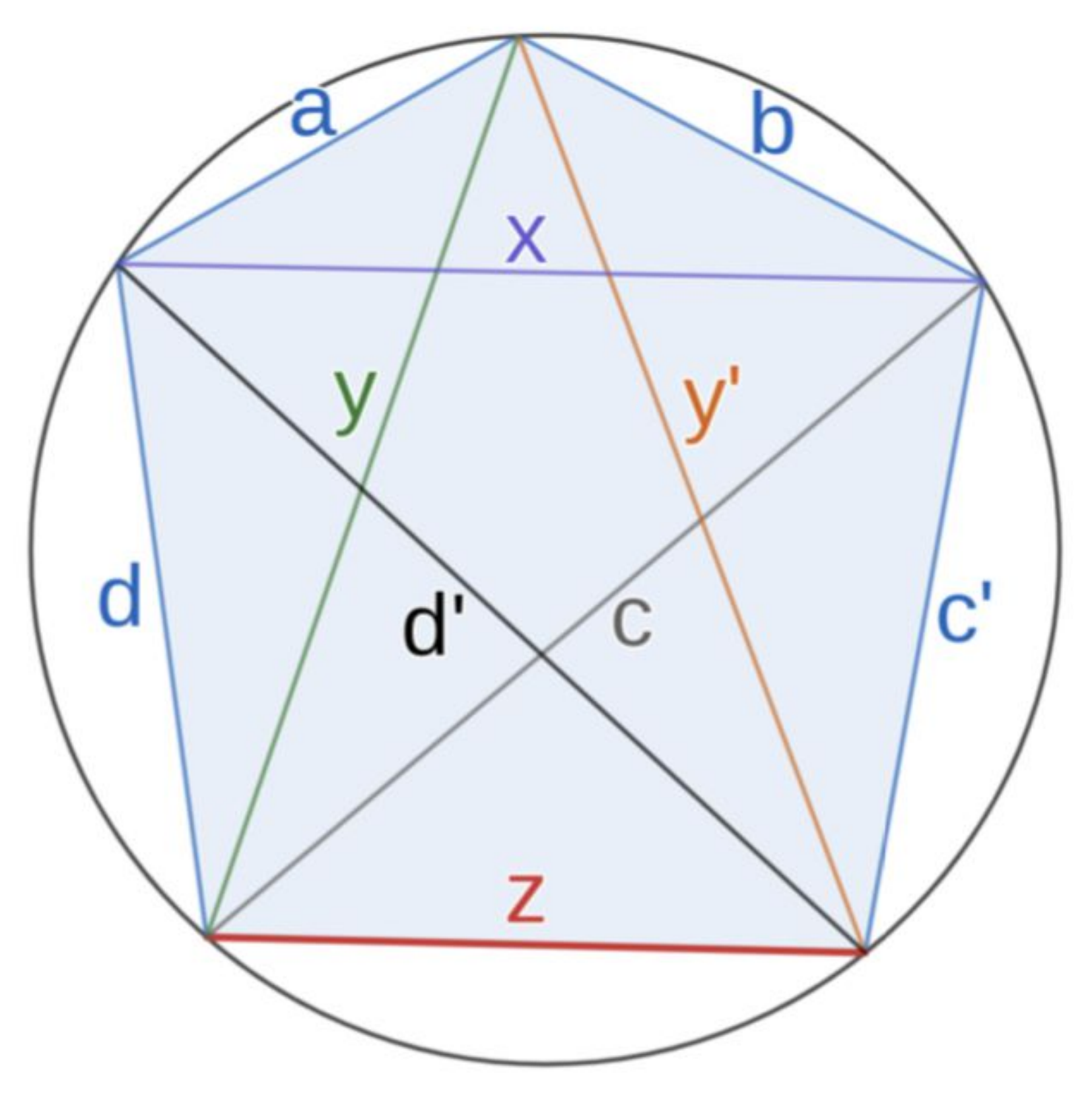}
\caption{Completing quadrilaterals to a pentagon}
\end{figure}

\subsection{Ptolemy yields Biedenharn--Elliot}

\begin{figure}[htp]
\centering
\includegraphics[width=.24\textwidth]{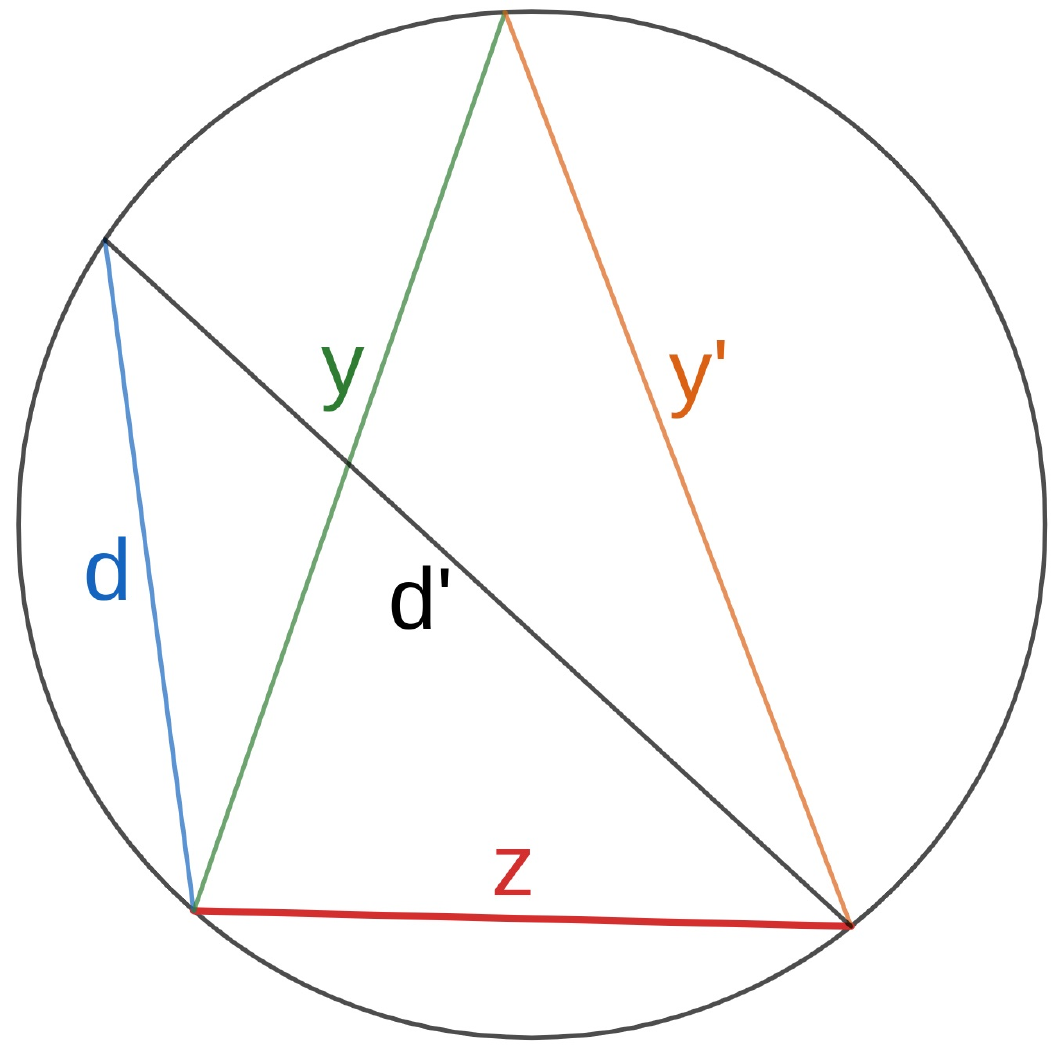}\hfill
\includegraphics[width=.24\textwidth]{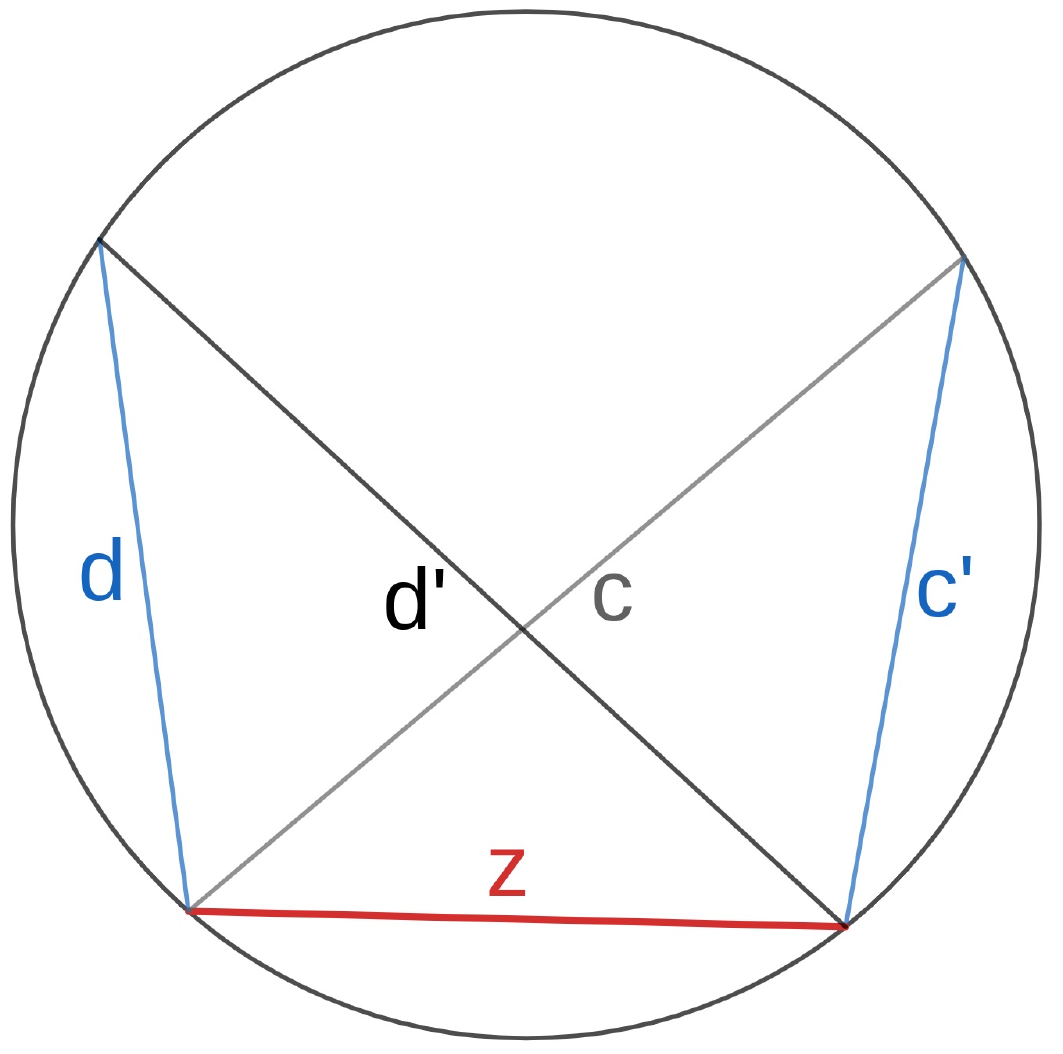}\hfill
\includegraphics[width=.24\textwidth]{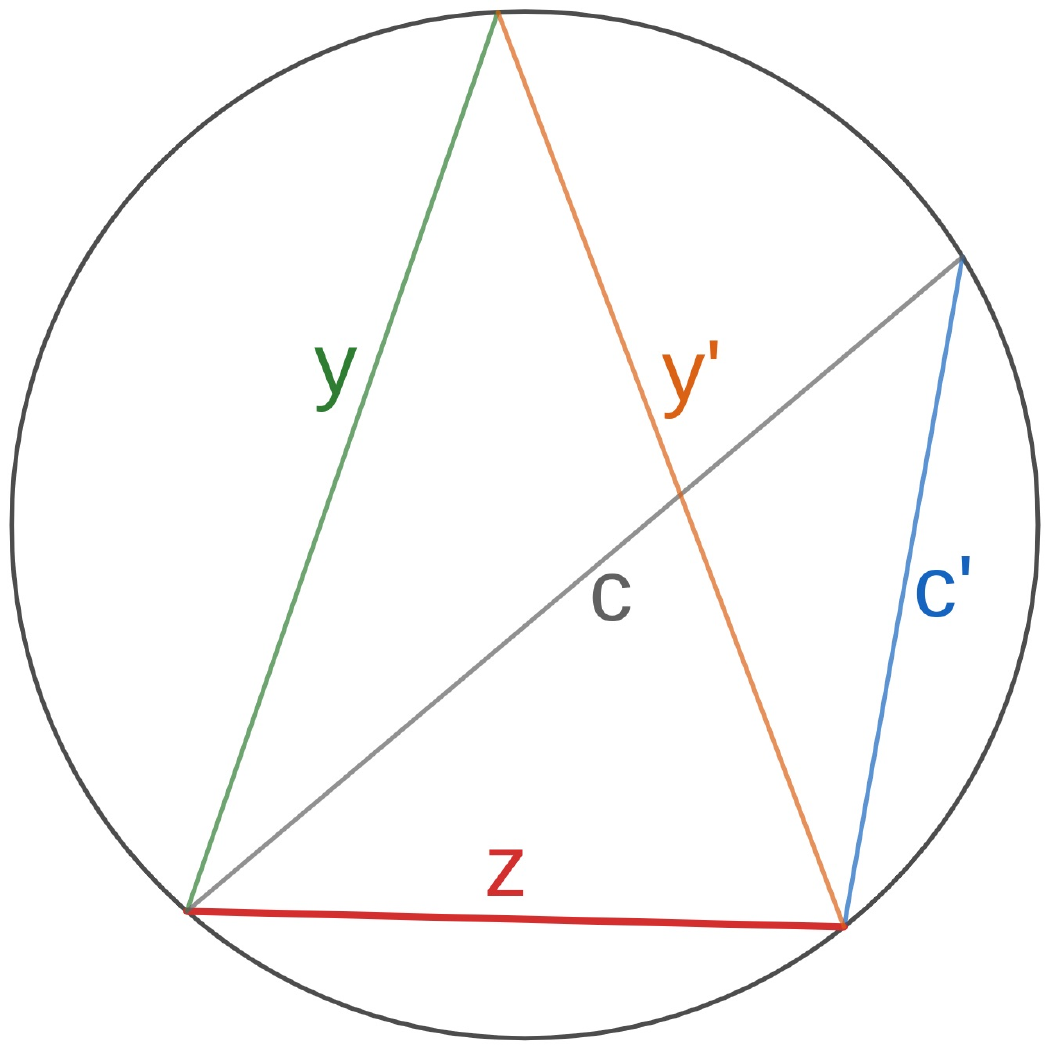}
\caption{}\label{fig:figure3}

\end{figure}

Associate real-valued variables to (oriented) $2$-cells of a
special cell.
With each true vertex we have a butterfly; associate this butterfly's
weight to it.

Let $S$ be a spine. Consider the commutative algebra
\[
P(S)=\Z_2\langle x\in S_2\mid \Pi_c, c\in S_0\rangle
\]
where $S_0$ is the set of vertices of the spine $S$, $S_2$ is the set of $2$-cells of the spine, and $\Pi_c$ is the Ptolemy relation~\eqref{eq:ptolemy} at the vertex $c$. Let $F(S)$ denote the fraction field of the algebra $P(S)$.

Two fields $F_1$ and $F_2$ of characteristic $2$ are \emph{stable equivalent} if there exist sets $X_1, X_2$ such that $F_1\otimes\Z_2(X_1)\simeq F_2\otimes\Z_2(X_2)$ where $\Z_2(X_1)$ is the field of rational functions over the generator set $X_i$, $i=1,2$.

\begin{theorem}
Given a 3-manifold $M$, the stable equivalence class of the fraction field $F(S)$ does not depend on the choice of the spine $S$ of the manifold $M$.
%The moduli space of solutions to a (modified) system of Biedenharn--Elliot equations taken over all true vertices of a special spine is an algebraic manifold whose type is an invariant of the initial 3-manifold.
\end{theorem}

\begin{proof}
Let $S$ and $S'$ be spines connected by a Matveev--Piergallini move (Fig.\ref{MP}).
%The moduli space of solutions on $S$ is characterized by the algebra $P(S)$ of functions on the space. The presentation of the algebra is
%\[
%P(S)=\Z_2\langle x\in S_2\mid \Pi_c, c\in S_0\rangle
%\]
%where $S_0$ is the set of vertices of the spine $S$, $S_2$ is the set of $2$-cells of the spine, and $\Pi_c$ is the Ptolemy relation~\eqref{eq:ptolemy} at the vertex $c$.
Let us show that the algebras $P(S)$ and $P(S')$ are stable equivalent. The correspondence between $2$-cells of the spines $S$ and $S'$ is one-to-one (except the cell $z$) and induces a homomorphism $\phi\colon F(S)\to F(S')$. We need to check that $\phi$ is well defined and bijective.

The algebras $P(S)$ and $P(S')$ have the same relations except those of the vertices participating in the Matveev--Piergallini move. Let us look at the those relations. For the algebra $P(S)$ we have
\[
ad+by+cx=0,\quad ad'+by'+c'x=0,
\]
hence
\[
x=\frac{ad'+by'}{c'},\quad y=\frac{ad+cx}b=\frac{acd'+ac'd+bcy'}{bc'}
\]
in the field $F(S)$. For the algebra $P(S')$ we have
\[
az+cy'+c'y=0,\quad bz+cd'+c'd=0,\quad dy'+d'y+xz=0,
\]
hence
\begin{gather*}
z=\frac{cd'+c'd}b,\\
y=\frac{az+cy'}{c'}=\frac{acd'+ac'd+bcy'}{bc'},\\
x=\frac{dy'+d'y}z=\frac{b(bc'dy'+ac(d')^2+ac'dd'+bcd'y')}{bc'(cd'+c'd)}=\frac{ad'+by'}{c'}
\end{gather*}
in $F(S')$. Thus, the expressions for $x$ and $y$ in $F(S)$ and $F(S')$ coincide. That means the map $\phi$ establishes an isomorphism between the fraction fields.

\begin{figure}[h]
\centering\includegraphics[width=0.8\textwidth]{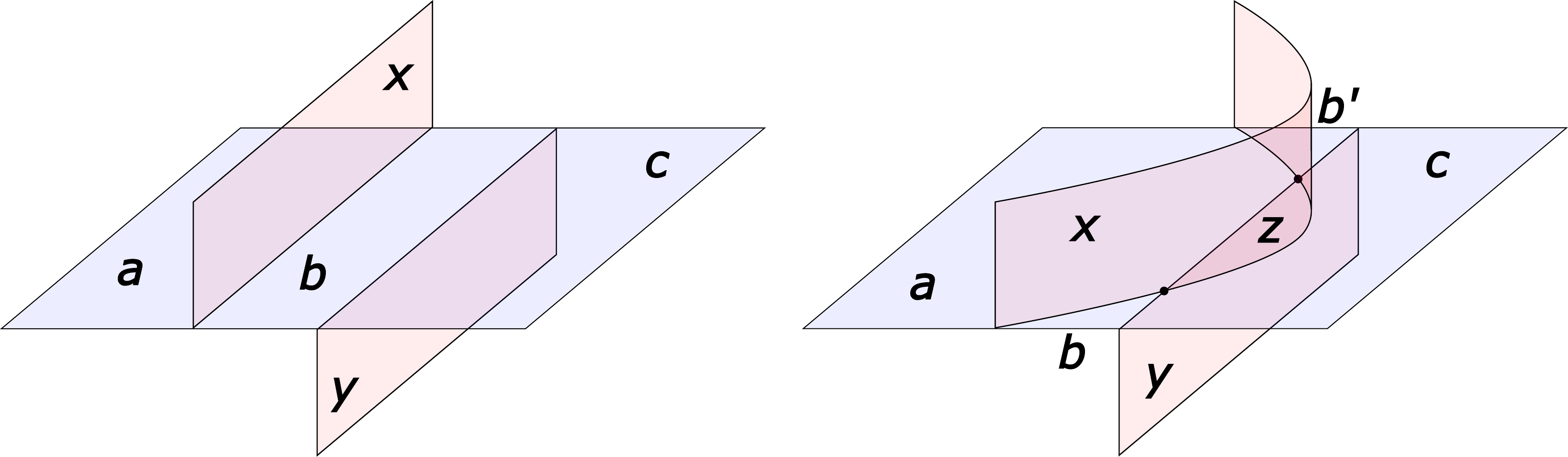}
\caption{The second Matveev-Piergallini move}
\label{MP2}
\end{figure}

For a second Matveev--Piergallini move (Fig.~\ref{MP2}) we have no relations for the spine $S$ and the relations
\[
ac+bz+xy=0,\quad ac+b'z+xy=0
\]
for the spine $S'$. Then $z=\frac{ac+xy}b$ and $b'=b$. Hence, we can identify the fields $F(S)$ and $F(S')$.

\begin{figure}[h]
\centering\includegraphics[width=0.7\textwidth]{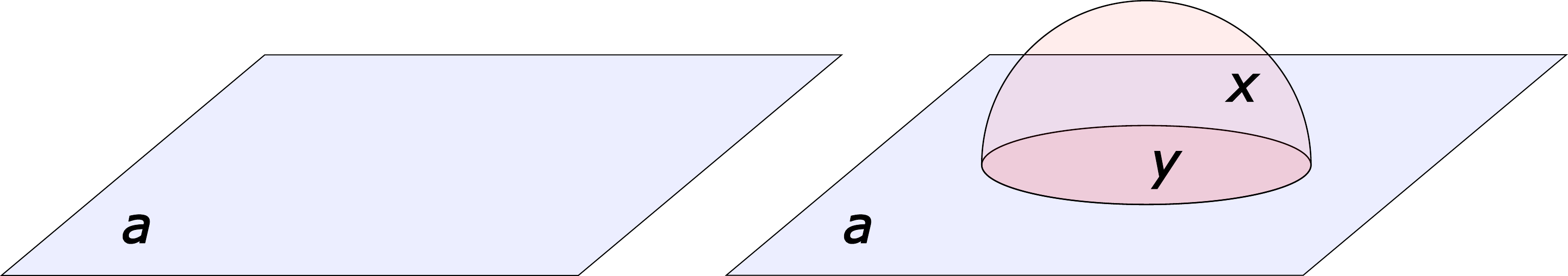}
\caption{The bubble move}
\label{bubble}
\end{figure}

For a bubble move (Fig.~\ref{bubble}) we get two additional variables and no new relations. Hence, $F(S')=F(S)\otimes \Z_2(x,y)$.

Thus, the fields $F(S)$ and $F(S')$ are stable equivalent.
\end{proof}

%\section{Volumes of tetrahedra and other solutions to Biedenharn--Elliot equations}
%
%Idea:
%instead of considering only weights $1,0$ corresponding to inscribed/non-inscribed polyhedra,
%let us do the following.
%
%With six weights $(a,b,x,c,d,y)$ we associate number $0$ if some of triangle inequalities fails
%for the sides $(a,b,x),(c,d,x),(a,d,y),(b,c,y)$.
%
%Otherwise there is a volume.
%
%$exp(vol)$.
%
%Matveev--Piergallini:
%
%$P+Q = F+G+H$ yields $exp(P)exp(Q)=exp(F)exp(G)exp(H)$.

\section{Generalisations of the above construction of 3-manifold invariants}\label{sect:generalisations}

In the above construction we dealt with a system of equations with 2-cell
weights acting as variables.

For each hexuple of labels we wrote an equation of the type

\begin{equation}\label{eq:ptolemy}
ac+bd = xy.
\end{equation}

Call it {\em the main equation}.

Let us fix some curvature $\kappa$ (positive, negative or zero) and let us
fix some radius $r$ (positive or equal to infinity).

By {\em $(\kappa,r)$-main equation} we mean the equation on the lengths
of $6$ edges of the tetrahedron inscribed in the circle of radius $r$ in the
3-space of constant curvature $\kappa$.
For the case $r=0$ we take the equation for an inscribed quadrilateral with
$2$ diagonals.

Hence, the $(0,0)$-main equation is exactly the above main equation
$ac+bd= xy$.

We argue that if we fix $\kappa,r$ and replace the main equation with the
$(\kappa,r)$-main equation, the space of solutions to the equations will
still be an invariant of $3$-manifolds.

Indeed, the only thing we used with the quadrilateral is:
whenever we draw two (inscribed) quadrilaterals sharing the same triangle
(refer to the figure) then we can find exactly one edge length (of the fifth
edge of the pentagon) which gives rise to the three remaining
quadrilaterals.

The only thing one has to change in the above argument is just to
replace ``inscribed in a circle on the plane'' with ``inscribed in a sphere
of radius $r$ in the space of constant curvature $\kappa$.''

\section{Why are the groups $\Gamma$ ubiquitous?}\label{sect:ubiquitous_gamma}

Many structures in mathematics are of extreme importance not because they just
solve a problem or provide nice invariants (which may also be the case) but because they
uncover close ties between various branches of mathematics. This allows
one to construct a ``universal translator'' between different languages different
mathematicians speak.

Among such famous objects, we have already mentioned cluster algebras
(see Fig.~\ref{fig:cl_alg}), we just mention Coxeter groups and Associahedra (Stasheff polytopes).
Among other branches of mathematics, the groups $\Gamma$ are related at
least to these three ones.

Among further directions which may appear in further papers, we mention just two ones:

In topology: passing from braids to knots, passing from 3-manifolds to 4-manifolds etc.

In group theory: deeper understanding the structure of the group $\Gamma$ themselves.

Among all possible presentations and actions of the groups $\Gamma_{n}^{k}$ we mention
just one method that the first named author calls the {\em photography method}:

{\em We take several (say, two) photos of some picture, and if we see that some of them are sufficient
to recover the whole picture, then we carefully write down the way one picture can
be transformed to the other.}

The Ptolemy transformation and the generalisation given in the last section are
evidences of that.

Namely, in order to show that {\em Ptolemy yields pentagon}, we need not calculate
anything (see~\cite{KM,FKMN,FMN,FKMN1}), we need not calculate anything:
we just say that if two quadrilaterals sharing three vertices are both inscribed then
the whole pentagon is inscribed, and the {\bf data} concerning any of its quadrilaterals
can be restored from the data of the two quadrilaterals.

What do we mean by {\bf data}? In geometry, those can be edge lengths,
angles, dihedral angles, areas, volumes, etc.

But probably other {\bf data} can come from interesting number theory
(say, quadruples, quintuples etc) and hence yield new invariants
of (at least) braids and 3-manifolds.

We expect that many representations of the groups $\Gamma$ and other groups (known or
still to be discovered) can be explained by using this method. Probably, many statements and relations the cluster algebra theory
also follow from this method.

\end{document}